 \documentclass[12pt]{article}
\usepackage{amsfonts}
\usepackage{graphicx}
 \usepackage{mathrsfs,amsfonts,amsmath}
 \usepackage{color}
 \makeatletter
 \@addtoreset{equation}{section}
 \makeatother
 \newtheorem{theorem}{Theorem}[section]
 \newtheorem{definition}[theorem]{Definition}
 \newtheorem{hypothesis}[theorem]{Hypothesis}
 \newtheorem{lemma}[theorem]{Lemma}
 \newtheorem{proposition}[theorem]{Proposition}
 \newtheorem{corollary}[theorem]{Corollary}
 
 \newtheorem{example}[theorem]{Example}

 \newtheorem{Prop}[theorem]{Proposition}

 \def\blemma{\begin{lemma}\sl{}\def\elemma{\end{lemma}}}
 \def\bproposition{\begin{proposition}\sl{}\def\eproposition{\end{proposition}}}
 \def\btheorem{\begin{theorem}\sl{}\def\etheorem{\end{theorem}}}

 \def\beqlb{\begin{eqnarray}}\def\eeqlb{\end{eqnarray}}
 \def\beqnn{\begin{eqnarray*}}\def\eeqnn{\end{eqnarray*}}

 \def\mbf{\mathbf}
 \def\mcr{\mathscr}
 \def\mbb{\mathbb}
 
 \def\const{\mathrm{const}}
 \def\proof{\noindent{\it Proof.~~}}
 \def\qed{\hfill$\Box$\medskip}

\newcommand{\bcen}{\begin{center}}
\newcommand{\ecen}{\end{center}}
\newcommand{\bgeqn}{\begin{equation}}
\newcommand{\edeqn}{\end{equation}}

\newcommand{\lan}{\langle}
\newcommand{\ran}{\rangle}

\def\az{\alpha}

\def\sz{\sigma}
\def\dz{\delta}
\def\lz{\lambda}

\def\ra{\rangle}
\def\la{\langle}

\begin{document}

\noindent\textit{(Version: Jan/04/2012)}

\bigskip\bigskip

\noindent{\Large\textbf{Branching Particle Systems in Spectrally}}

\smallskip

\noindent{\Large\textbf{One-sided L\'evy Processes} \footnote{Supported
by NSFC (11071021, 11131003, 11126037), 985 Project and NSERC grants.}}

\bigskip

\noindent{Hui He, Zenghu Li and Xiaowen Zhou}

\bigskip

\noindent{\it Beijing Normal University and Concordia University}

\bigskip\bigskip

{\narrower{\narrower

\centerline{\bf Abstract}

\bigskip

We investigate the branching structure coded by the excursion above zero
of a spectrally positive L\'evy process. The main idea is to identify the
level of the L\'evy excursion as the time and count the number of jumps
upcrossing the level. By regarding the size of a jump as the birth site of
a particle, we construct a branching particle system in which the
particles undergo nonlocal branchings and deterministic spatial motions to
the left on the positive half line. A particle is removed from the system
as soon as it reaches the origin. Then a measure-valued Borel right Markov
process can be defined as the counting measures of the particle system.
Its total mass evolves according to a Crump-Mode-Jagers branching process
and its support represents the residual life times of those existing
particles. A similar result for spectrally negative L\'evy process is
established by a time reversal approach. Properties of the measure-valued
processes can be studied via the excursions for the corresponding L\'evy
processes.

\bigskip\bigskip


\bigskip


\par}\par}

%%%%%%%%%%%%%%%%%%%%%%%%%%%%%%%%%%%%%%%

\section{Introduction}

\setcounter{equation}{0}

Branching processes embedded in processes with independent increments have
been studied by many authors. The study yields detailed information and
understandings in the two classes of processes. In particular, Dwass
\cite{[Dw75]} constructed branching processes from simple random walks. To
study random walks in random environment Kesten et al \cite{[KKS75]}
constructed a Galton-Watson process with geometric offspring law from a
simple random walk. Multitype branching processes have also been
introduced in the study of random walks in random environment; see
\cite{[HW09], [HZ10], [Key87]} and the references therein. Since
continuous state branching processes and Brownian motions arise as the
scaling limits of Galton-Watson processes and simple random walks,
respectively, we may naturally expect some branching structures embedded
in a Brownian motion. The well-known Knight-Ray theorem brings an answer
to this question; see also \cite{[Le89], [NP89]}.

Le Gall and Le Jan \cite{[LL98a], [LL98b]} recovered a deep connection
between general continuous state branching processes and spectrally
positive L\'evy processes. Furthermore, Duquesne and Le Gall
\cite{[DuL02], [DuL05]} showed that the branching points of a L\'evy tree
constructed in \cite{[LL98a]} are of two types: binary nodes (i.e. vertex
of degree three), which are given by the Brownian part of the L\'evy
process, and infinite nodes (i.e. vertex of infinite degrees), which are
given by the jumps of the L\'evy process. The size of the jump is also
called the size of the corresponding infinite node (or the mass of the
forest attached to the node).

In the interesting recent work \cite{[L10]}, Lambert used spectrally
positive L\'evy processes for the first time to code random splitting
trees. In the population dynamics represented by the splitting tree, the
number of individuals evolves according to a binary Crump-Mode-Jagers
process. It was proved in \cite{[L10]} that the contour process of the splitting
tree truncated up to a certain level is a spectrally positive L\'evy
process reflected below this level and killed upon hitting zero. From this
result Lambert derived a number of properties of the splitting tree and
the Crump-Mode-Jagers process.

The purpose of this paper is to give a formulation of the branching
structures of spectrally one-sided L\'evy processes in terms of
measure-valued processes, which we call \textit{single-birth branching
particle systems}. Those structures are undoubtedly conveyed by the
random splitting trees, so we could have derived the results from those
of Lambert \cite{[L10]}. However, we think a simple construction of the
branching particle systems directly from the L\'evy process is of
interest. In addition, we show that the branching systems are Borel right
Markov processes in a suitable state space and characterize their
transition semigroups using some simple quasi-linear integral equations.
Those properties make the branching systems  easier to handle
than the Crump-Mode-Jagers processes. A more precise description of the
branching structures is given in the next paragraph.

Let us consider a typical trajectory of the spectrally positive L\'evy
process with negative drift $\{S_t: t\ge 0\}$ started from $a>0$ and
killed upon hitting zero; see Figure $1'$. Let $\{y_i: i=1,2,3\}$ denote
the sizes of jumps. Then the sample path of a branching particle system
can be obtained in the following way: At time zero, an ancestor starts
off from $a>0$ and moves toward the left at the unit speed. At times
$z_1$ and $z_3$, it gives birth to two children at positions $y_1$ and
$y_3$, respectively. At time $z_2$, the first child of the ancestor gives
birth to a child at position $y_2$. Once an individual hits zero, it is
removed from the system. So the ancestor dies at time $a$ and its two
children die at times $z_1+y_1$ and $z_3+y_3$, respectively.

From the structures described above, we use a time reversal to derive a
similar result for spectrally negative L\'evy processes with positive
drift. We will see that the branching systems we encounter here are
actually very special cases of the models studied in \cite{[DGL02],
[Li11]}. Unfortunately, by now we can only treat L\'evy processes with
bounded variations as in \cite{[L10]}. An interesting open question is to
give a description of the branching structures of general spectrally
one-sided L\'evy processes in terms of measure-valued branching
processes. We hope to see the precise formulation of such structures in
the future.

\includegraphics[width=0.8\textwidth,height=0.3\textheight]{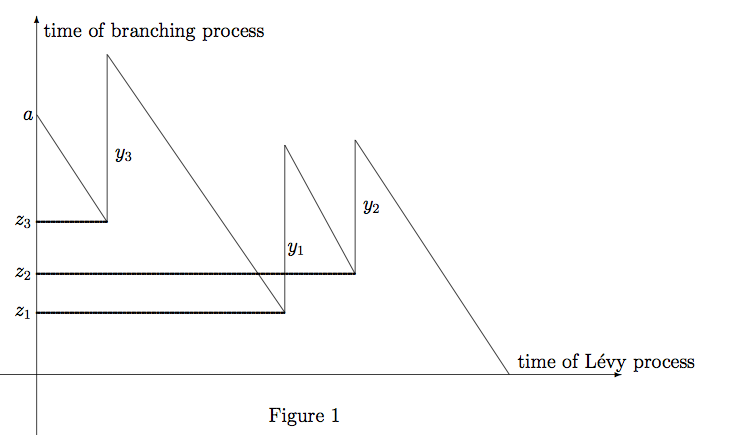}

The rest of this paper is arranged as follows. In Section 2, we introduce
some branching particle systems on the positive half line involving
nonlocal branching structures. In Section 3, we extend the model to the
case with infinite branching rates. In Section 4 the result on the
branching structures in spectrally positive L\'evy processes with negative
drift is established. In Section 5, we derive the branching structures for
spectrally negative L\'evy processes with positive drift by a time
reversal approach. Some properties of our branching systems are studied in
Section 6. In Section 7, we discusses briefly the connection of the
branching systems with the Crump-Mode-Jagers models.

\medskip

\textbf{Notations.} Write $\mbb{R}_+ = [0,\infty)$. Given a metric space
$E$, we denote by $B(E)$ the Banach space of bounded Borel functions on
$E$ endowed with the supremum/uniform norm ``$\|\cdot\|$''. Let $C(E)$ be
the subspace of $B(E)$ consisting of bounded continuous functions on $E$.
We use the superscript ``$+$'' to denote the subset of positive elements
of the function spaces, e.g., $B^+(\mbb{R}_+)$ and $C^+(0,\infty)$. Let
$M(E)$ denote the space of finite Borel measures on $E$ endowed with the
topology of weak convergence. Let $N(E)$ be the set of integer-valued
measures in $M(E)$. For a measure $\mu$ and a function $f$ on $E$ write
$\la\mu,f\ra = \int fd\mu$ if the integral exists. Other notations will
be explained when they first appear.

%%%%%%%%%%%%%%%%%%%%%%%%%%%%%%%%%%%%%%%%%%%%%%

\section{Branching systems on the positive half line}

\setcounter{equation}{0}

We begin with the description of a branching system of particles on
$\mbb{R}_+$. Suppose that $\alpha>0$ is a constant, $\eta = \eta(dx)$ is
a probability measure on $(0,\infty)$ and $g=g(z)$ is a probability
generating function with $g'(1)< \infty$. Let $\{\xi_t: t\ge 0\}$ be the
$\mbb{R}_+$-valued Markov process defined by $\xi_t:=(\xi_0-t)\vee 0$.
Let $F(0,\cdot)$ be the unit mass at $\delta_{0}\in N (\mbb{R}_+)$. For
$x>0$, let $F(x,\cdot)$ be the distribution on $N(\mbb{R}_+)$ of the
random measure
 \beqnn
\delta_{x}+\sum_{i=1}^Z\delta_{Y_i},
 \eeqnn
where $Z$ is an integer-valued random variable with distribution
determined by $g=g(z)$ and $\{Y_1, Y_2, \cdots\}$ are i.i.d random
variables on $(0,\infty)$ with distribution $\eta(dx)$. Here we assumed
that $Z$ and $\{Y_1,Y_2,\cdots\}$ are independent.

Suppose that we have a set of particles on $\mbb{R}_+$ moving
independently according to the law of $\{\xi_t: t\ge 0\}$. A particle is
frozen as soon as it reaches zero. Before that at each
$\alpha$-exponentially distributed random time, the particle gives birth
to a random number of offspring according to the law specified by the
generating function $g=g(z)$, and those offspring are scattered over
$\mbb{R}_+$ independently according to the distribution $\eta(dx)$. It is
assumed as usual that the reproduction of different particles are
independent of each other. Let $\bar{X}_t(B)$ denote the number of
particles in the set $B\in \mcr{B}(\mbb{R}_+)$ at time $t\ge 0$. By
Dawson et al \cite[p.103]{[DGL02]} one can see that $\{\bar{X}_t: t\ge
0\}$ is a Markov process on $N(\mbb{R}_+)$ with transition semigroup
$(\bar{Q}_t)_{t\ge 0}$ defined by
 \beqlb\label{2.1}
\int_{N(0,\infty)} e^{-\la\nu,f\ra} \bar{Q}_t(\mu,d\nu)
 =
\exp\{-\la\mu,\bar{U}_t f\ra\}, \qquad f\in B^+(\mbb{R}_+),
 \eeqlb
where $(t,x)\mapsto \bar{U}_tf(x)$ is the unique positive
solution of
 \beqnn
e^{-\bar{U}_t f(x)}
 &=&
e^{-f((x-t)\vee 0)} - \alpha\int_0^t e^{-\bar{U}_{t-s}f((x-s)\vee
0)} d s + \alpha\int_0^t e^{-\bar{U}_{t-s}f(0)}1_{\{x\le s\}} d
s \cr
 &&\qquad\qquad
+\, \alpha\int_0^t e^{-\bar{U}_{t-s}f(x-s)}1_{\{x>s\}} g(\la\eta,
e^{-\bar{U}_{t-s}f}\ra) d s;
 \eeqnn
see also Dawson et al \cite[pp.95-96]{[DGL02]} and Li
\cite[p.98]{[Li11]}. By Proposition~2.9 of \cite{[Li11]}, the above
equation can be rewritten as
 \beqlb\label{2.2}
e^{-\bar{U}_t f(x)}
 =
e^{-f((x-t)\vee 0)} - \alpha\int_0^t
e^{-\bar{U}_{t-s}f(x-s)}1_{\{x>s\}} \big[1 - g(\la\eta,
e^{-\bar{U}_{t-s}f}\ra)\big] ds.
 \eeqlb
By Proposition~A.49 of \cite{[Li11]}, for $f\in B(\mbb{R}_+)$ there is a
unique locally bounded solution $(t,x)\mapsto \bar{\pi}_tf(x)$ to the
equation
 \beqlb\label{2.3}
\bar{\pi}_t f(x) = f((x-t)\vee 0) + \alpha g'(1)\int_0^t 1_{\{x>s\}}\la
\eta,\bar{\pi}_{t-s} f\ra ds.
 \eeqlb
Moreover, the linear operators $(\bar{\pi}_t)_{t\ge 0}$ on $B(\mbb{R}_+)$
form a semigroup and
 \beqlb\label{2.4xx}
\|\bar{\pi}_tf\|\le \|f\|e^{\alpha g'(1)t}, \qquad t\ge 0.
 \eeqlb

\bproposition\label{t2.1} For $t\ge 0$ and $f\in B^+(\mbb{R}_+)$ we have
$\bar{U}_tf\le \bar{\pi}_tf$ and for $t\ge 0$ and $f\in B^+(\mbb{R}_+)$
we have
 \beqlb\label{2.4}
\int_{N(\mbb{R}_+)}\la\nu,f\ra \bar{Q}_t(\mu,d\nu)
 =
\la\mu,\bar{\pi}_tf\ra.
 \eeqlb
\eproposition

\proof For $t\ge 0$ and $f\in B^+(\mbb{R}_+)$ one can use (\ref{2.2}) and
(\ref{2.3}) to see
 \beqnn
\bar{\pi}_t f(x) = \frac{\partial}{\partial \theta}
\bar{U}_t(\theta f)(x)\Big|_{\theta=0}.
 \eeqnn
Then (\ref{2.4}) follows by differentiating both sides of (\ref{2.2}). By
(\ref{2.1}), (\ref{2.4}) and Jensen's inequality it is clear that
$\bar{U}_tf(x)\le \bar{\pi}_tf(x)$ for $x\ge 0$. By linearity we also
have (\ref{2.3}) and (\ref{2.4}) for $f\in B(\mbb{R}_+)$. \qed

\bproposition\label{t2.3} For any $f\in B^+(\mbb{R}_+)$ the mapping
$t\mapsto \bar{U}_tf(\cdot+t)$ from $[0,\infty)$ to $B^+(\mbb{R}_+)$ is
increasing and locally Lipschitz in the supremum norm. Moreover, for any
$t\ge r\ge 0$ we have
 \beqlb\label{2.5}
0\le e^{-\bar{U}_rf(x+r)} - e^{-\bar{U}_t f(x+t)}\le \alpha (t-r).
 \eeqlb
\eproposition

\proof For any $t,x\ge 0$ one can use (\ref{2.2}) to see
 \beqnn
e^{-\bar{U}_t f(x+t)}
 =
e^{-f(x\vee 0)} - \alpha\int_0^t e^{-\bar{U}_sf(x+s)}1_{\{x+s>0\}} \big[1
- g(\la\eta, e^{-\bar{U}_sf}\ra)\big] ds.
 \eeqnn
Then we have (\ref{2.5}). Since $t\mapsto \bar{U}_tf$ is locally bounded
by Proposition~\ref{t2.1}, we see $t\mapsto \bar{U}_tf(\cdot+t)$ is
increasing and locally Lipschitz in the supremum norm. \qed

\bproposition\label{t2.4} For any $f\in B^+(\mbb{R}_+)$ the function
$(t,x)\mapsto \bar{U}_tf(x)$ is the unique locally bounded positive
solution of
 \beqlb\label{2.6}
\bar{U}_tf(x) = f((x-t)\vee0) + \az\int_0^t 1_{\{x>s\}} \big[1 -
g(\la\eta,e^{-\bar{U}_{t-s}f}\ra)\big]ds.
 \eeqlb
\eproposition

\proof For notational convenience, in this proof we set $f(x) = f(0)$ and
$\bar{U}_tf(x) = \bar{U}_tf(0)$ for all $x\le 0$ and $t\ge 0$. Let
$0=t_0<t_1<\cdots<t_n=t$ be a partition of $[0,t]$. For $x\in\mbb{R}_+$,
we can write
 \beqlb\label{2.7}
\bar{U}_tf(x) = f(x-t) + \sum_{i=1}^n \Big[\bar{U}_{t-t_{i-1}}
f(x-t_{i-1}) - \bar{U}_{t-t_i}f(x-t_i)\Big].
 \eeqlb
Note that Proposition \ref{t2.4} implies $\bar{U}_{t-t_{i-1}}f(x-t_{i-1})
- \bar{U}_{t-t_i}f(x-t_i)\geq0$. By (\ref{2.2}), (\ref{2.5}) and Taylor's
formula, as $t_i-t_{i-1}\to 0$,
 \beqnn
&&\Big[\bar{U}_{t-t_{i-1}}f(x-t_{i-1}) - \bar{U}_{t-t_i}f(x-t_i)\Big] \cr
 &&\qquad
= e^{\bar{U}_{t-t_{i-1}}f(x-t_{i-1})} \Big[
e^{-\bar{U}_{t-t_i}f(x-t_i)} - e^{-\bar{U}_{t-t_{i-1}}
f(x-t_{i-1})}\Big] + o(t_i-t_{i-1}) \cr
 &&\qquad
= \int_0^{t_i-t_{i-1}} \big[1+\varepsilon_i(s,x)\big]
1_{\{x-t_{i-1}>s\}} \big[1 - g(\la \eta,
e^{-\bar{U}_{t-t_{i-1}-s}f}\ra)\big] ds + o(t_i-t_{i-1}),
 \eeqnn
where
 \beqnn
\varepsilon_i(s,x)
 =
e^{\bar{U}_{t-t_{i-1}}f(x-t_{i-1})} \Big[e^{-\bar{U}_{t-t_{i-1}-s}
f(x-t_{i-1}-s)} - e^{-\bar{U}_{t-t_{i-1}} f(x-t_{i-1})}\Big].
 \eeqnn
By Propositions~\ref{t2.1} and~\ref{t2.3} one can see that
 \beqnn
0\le \varepsilon_i(s,x)
 \le
\alpha (t_i-t_{i-1}) \exp\big\{\|f\|e^{\alpha g^\prime(1)t}\big\},
\qquad 0\le s\le t_i-t_{i-1}.
 \eeqnn
It then follows that
 \beqnn
&&\Big[\bar{U}_{t-t_{i-1}}f(x-t_{i-1})-\bar{U}_{t-t_i}f(x-t_i)\Big] \cr
 &&\qquad
= \int_0^{t_i-t_{i-1}} 1_{\{x-t_{i-1}>s\}} \big[1 - g(\la\eta,\,
e^{-\bar{U}_{t-t_{i-1}-s}f}\ra)\big] ds + o(t_i-t_{i-1}) \cr
 &&\qquad
= \int_{t_{i-1}}^{t_i} 1_{\{x>s\}} \big[1 - g(\la\eta,\,
e^{-\bar{U}_{t-s} f}\ra)\big] ds + o(t_i-t_{i-1}).
 \eeqnn
Substituting this into (\ref{2.7}) and letting $\max_{1\le i\le n}
(t_i-t_{i-1})\to 0$ we obtain (\ref{2.6}). The uniqueness of the solution
of the equation follows from Proposition~2.18 in \cite{[Li11]}. \qed

\btheorem\label{t2.1xx} There is a Borel right transition semigroup
$(Q_t)_{t\ge 0}$ on $N(0,\infty)$ defined by
 \beqlb\label{2.10}
\int_{N(0,\infty)}e^{-\la\nu,f\ra}Q_t(\mu,d\nu)
 =
e^{-\la\mu,U_tf\ra},\qquad f\in B^{+}(0,\infty),
 \eeqlb
where $(t,x)\mapsto U_tf(x)$ is the unique locally bounded positive
solution of
 \beqlb\label{2.11}
U_tf(x) = f(x-t)1_{\{x>t\}} + \az\int_0^t 1_{\{x> t-s\}} [1 -
g(\la\eta,e^{-U_sf}\ra)]ds, \quad t\ge 0, x>0.
 \eeqlb
\etheorem

\proof It is not hard to see that (\ref{2.11}) is a special cases of
(2.21) in \cite[p.39]{[Li11]}. By (\ref{2.2}) we have $\bar{U}_t f(0) =
f(0)$ for all $t\ge 0$. Consequently, if $\{\bar{X}_t: t\ge 0\}$ is a
Markov process with transition semigroup $(\bar{Q}_t)_{t\ge 0}$ defined
by (\ref{2.1}) and (\ref{2.6}), then $\{\bar{X}_t|_{(0,\infty)}: t\ge
0\}$ is a Markov process in $N(0,\infty)$ with transition semigroup
$(Q_t)_{t\ge 0}$ defined by (\ref{2.10}) and (\ref{2.11}). By
Theorem~5.12 of \cite{[Li11]}, we can extend $(Q_t)_{t\ge 0}$ to a Borel
right semigroup on the space of finite measures on $(0,\infty)$. Then
$(Q_t)_{t\ge 0}$ itself is a Borel right semigroup. \qed

By Proposition~\ref{t2.1} we have the following:

\bproposition\label{t2.5} For every $f\in B(0,\infty)$ there is a unique
locally bounded solution $(t,x)\mapsto \pi_tf(x)$ of
 \beqlb\label{2.10aa}
\pi_t f(x) = f((x-t)\vee 0) + \alpha g'(1)\int_0^t 1_{\{x>s\}}\la
\eta,\pi_{t-s} f\ra ds.
 \eeqlb
Moreover, the linear operators $(\pi_t)_{t\ge 0}$ on $B(0,\infty)$ form a
semigroup and
 \beqlb\label{2.11aa}
\int_{N(0,\infty)}\la\nu,f\ra Q_t(\mu,d\nu)
 =
\la\mu,\pi_tf\ra, \qquad t\ge 0, f\in B(0,\infty).
 \eeqlb
\eproposition

\bproposition\label{t2.6} We have $U_tf(x)\le \pi_tf(x)\le \|f\|e^{\alpha
g'(1)t}$ for $t\ge 0, x>0$ and $f\in B(0,\infty)$. \eproposition

A Markov process in $N(0,\infty)$ with transition semigroup $(Q_t)_{t\ge
0}$ defined by (\ref{2.10}) and (\ref{2.11}) will be referred to as a
\textit{branching system} of particles with parameters $(g,\alpha,\eta)$,
where $g$ is the \textit{generating function}, $\alpha$ is the
\textit{branching rate} and $\eta$ is the \textit{offspring position
law}.

%%%%%%%%%%%%%%%%%%%%%%%%%%%%%%%%%%%%%%%

\section{The system with infinite branching rate}

\setcounter{equation}{0}

In this section, we consider a system of particles, which can be thought
of as a branching system with infinite branching rate. Let $\rho(x) = x$
for $x\in(0,\infty)$. Let $B_\rho(0,\infty)$ be the set of Borel functions
on $(0,\infty)$ bounded by $\rho\cdot \const$. Let $C_\rho(0,\infty)$ be
the subset of $B_\rho(0,\infty)$ consisting of continuous functions. Let
$M_\rho(0,\infty)$ be the set of Borel measures $\mu$ on $(0,\infty)$
satisfying $\la\mu,\rho\ra< \infty$. Let $N_\rho(0,\infty)$ be the set of
integer-valued measures in $M_\rho(0,\infty)$. We endow $M_\rho(0,\infty)$
and $N_\rho(0,\infty)$ with the topologies defined by the convention that
 \beqlb\label{6.1}
\mu_n\to \mu ~ \mbox{if and only if} ~ \la\mu_n,f\ra\to \la\mu,f\ra ~
\mbox{for all} ~ f\in C_\rho(0,\infty).
 \eeqlb
We say a function $(t,x)\mapsto u_t(x)$ on $[0,\infty)\times (0,\infty)$
is \textit{locally $\rho$-bounded} if
 \beqnn
\sup_{0\le s\le t}\sup_{x\in (0,\infty)} |\rho(x)^{-1}u_s(x)|<
\infty, \qquad t\ge 0.
 \eeqnn

Let $c>0$ be a constant and let $\Pi(dz)$ be a $\sigma$-finite measure on
$(0,\infty)$ such that $\la\Pi,\rho\ra<c$. Given $f\in
B_\rho^+(0,\infty)$, we consider the following evolution equation:
 \beqlb\label{6.2a}
U_tf(x)=f(x-t)1_{\{x>t\}} + c^{-1}\int_0^t 1_{\{x>s\}} \la\Pi,
1-e^{-U_{t-s}f}\ra ds.
 \eeqlb

\blemma\label{t6.1} For each $f\in B_\rho^+(0,\infty)$ there is at most
one locally $\rho$-bounded positive solution of (\ref{6.2a}).
 \elemma

\proof Suppose that $(t,x)\mapsto U_tf(x)$ and $(t,x)\mapsto V_tf(x)$ are
two locally $\rho$-bounded solutions of (\ref{6.2a}). Let
 \beqnn
l_T(x) = \sup_{0\le t\le T} |\rho(x)^{-1}(U_tf(x) - V_tf(x))|.
 \eeqnn
Then for any $0\leq t\leq T$ we have
 \beqnn
|U_tf(x) - V_tf(x)|
 &\le&
c^{-1}\int_0^t1_{\{x>s\}}\la\Pi, |e^{-U_{t-s}f} - e^{-V_{t-s}f}|\ra ds
\cr
 &\le&
c^{-1}\int_0^t1_{\{x>s\}}\la\Pi, |U_{t-s}f - V_{t-s}f|\ra ds \cr
 &\le&
c^{-1}\int_0^t1_{\{x>s\}} ds\|l_T\|\la\Pi,\rho\ra
 \le
c^{-1}\rho(x)\|l_T\|\la\Pi,\rho\ra,
 \eeqnn
which implies $\|l_T\|\le c^{-1}\|l_T\|\la\Pi,\rho\ra$. Then we have
$\|l(T)\|=0$ as $\la\Pi,\rho\ra<c$. \qed

\bproposition\label{t6.2} For each $f\in B_\rho^+(0,\infty)$, there is a
unique locally $\rho$-bounded positive solution $(t,x)\mapsto U_tf(x)$ of
(\ref{6.2a}) and the solution is increasing in $(\Pi, f)\in
M_\rho(0,\infty)\times B_\rho^+(0,\infty)$. Furthermore, the operators
$(U_t)_{t\ge 0}$ on $B_\rho^+(0,\infty)$ form a semigroup and
 \beqlb\label{6.3}
\|\rho^{-1}U_tf\|\le (c - \la\Pi,\rho\ra)^{-1}\|\rho^{-1}f\|, \qquad t\ge
0.
 \eeqlb
\eproposition

\proof \textit{Step 1)} We first assume that $\Pi\in M(0,\infty)$ and
$f\in B^+(0,\infty)$. By Theorem~\ref{Thmpoisson2} there is a unique
locally bounded positive solution $(t,x)\mapsto U_tf(x)$ of (\ref{6.2a}).
This solution can also be constructed by a simple iteration procedure. In
fact, if we let $u_0(t,x) = 0$ and define $u_n(t,x) = u_n(t,x,f)$
inductively by
 \beqlb\label{t6.2a1}
u_n(t,x) &:=& f(x-t)1_{\{x>t\}} + c^{-1}\int_0^t 1_{\{x> s\}}ds
\int_0^{\infty} [1 - e^{-u_{n-1}(t-s,z)}] \Pi(dz),
 \eeqlb
then $u_n(t,x)\to U_tf(x)$ increasingly as $n\to \infty$; see
Proposition~2.18 of \cite{[Li11]}. Using this construction one can see
that the solution of (\ref{6.2a}) is increasing in $ (\Pi, f)\in
M(0,\infty)\times B^+(0,\infty)$.

\medskip

\textit{Step 2)} Next, we assume that $\Pi\in M(0,\infty)$ and $f\in
B^+_{\rho}(0,\infty)$. Let $f_k = f\land k$ for $k\ge 1$. Let
$(t,x)\mapsto U_tf_k(x)$ be the unique locally bounded positive solution
of (\ref{6.2a}) with $f$ replaced by $f_k$. According to the argument
above the sequence $\{U_tf_k\}$ is increasing in $k\ge 1$. By
(\ref{6.2a}) and Proposition~\ref{t2.6} we have
 \beqnn
U_tf_k(x) &\le& \|\rho^{-1}f_k\|\rho(x) + c^{-1}\int_0^t 1_{\{x> s\}}ds
\int_{\mbb R+}U_{t-s}f_k(z)\Pi(dz) \cr
 &\le&
\Big[\|\rho^{-1}f_k\| + c^{-1}\|f_k\|\la\Pi, 1\ra\exp\{c^{-1}\la\Pi,1\ra
t\}\Big] \rho(x).
 \eeqnn
Thus $(t,x)\mapsto U_tf_k(x)$ is locally $\rho$-bounded. On the other
hand, if we set
 \beqnn
l_k(t,x) := \sup_{0\le s\le t}U_sf_k(x),
 \eeqnn
then
 \beqnn
l_k(t,x) &\le& \|\rho^{-1}f_k\|\rho(x) + c^{-1}\rho(x)\sup_{0\le s\le
t}\int_{(0,\infty)}U_{s}f_k(z)\Pi(dz) \cr
 &\le&
\Big[\|\rho^{-1}f\| + c^{-1}\|\rho^{-1}l_k(t)\|\la\Pi, \rho\ra\Big]
\rho(x).
 \eeqnn
It follows that
 \beqnn
\rho(x)^{-1}l_k(t,x) &\le& \|\rho^{-1}f\| + c^{-1}
\|\rho^{-1}l_k(t)\|\la\Pi, \rho\ra,
 \eeqnn
which implies
 \beqlb\label{eqn6.1a}
\|\rho^{-1}l_k(t)\| \le \frac{\|\rho^{-1}f\|}{1 - c^{-1}\la\Pi, \rho\ra}
 =
\frac{c\|\rho^{-1}f\|}{c - \la\Pi, \rho\ra}.
 \eeqlb
In particular, we have
 \beqnn
\|\rho^{-1}U_tf_k\| \le c\|\rho^{-1}f\|(c - \la\Pi,\rho\ra)^{-1}, \qquad
t\ge 0.
 \eeqnn
Then the limit $U_tf(x) := \lim_{k\to \infty} U_tf_k(x)$ exists. It is
easy to see that $(t,x)\mapsto U_tf(x)$ is a locally $\rho$-bounded
positive solution of (\ref{6.2a}) satisfying (\ref{6.3}).

\medskip

\textit{Step 3)} In the general case, let $\Pi_k(dz) = 1_{\{z\ge
1/k\}}\Pi(dz)$ for $k\ge 1$. For $f\in B^+(0,\infty)$ let $(t,x)\mapsto
U_t^{(k)}f(x)$ be the unique locally $\rho$-bounded positive solution of
(\ref{6.2a}) with $\Pi$ replaced by $\Pi_k$. By the second step, we can
define $U_t^{(k)}f$ by the equation for any $f\in B^+_{\rho}(0,\infty)$.
The sequence $\{U_t^{(k)}f\}$ is increasing by the first and the second
steps. As in the second step one can see the limit $U_tf(x) := \lim_{k\to
\infty} U_t^{(k)}f(x)$ exists and is a locally $\rho$-bounded positive
solution of (\ref{6.2a}) satisfying (\ref{6.3}). The uniqueness of the
solution follows from Lemma~\ref{t6.1}, which yields the semigroup
property of $(U_t)_{t\geq0}$. \qed

\bproposition\label{t6.3} For each $f\in B_\rho(0,\infty)$, there is a
unique locally $\rho$-bounded solution $(t,x)\mapsto \pi_tf(x)$ of
 \beqlb\label{6.5}
\pi_tf(x)=f(x-t)1_{\{x>t\}} + c^{-1}\int_0^t 1_{\{x>t-s\}}
\la\Pi,\pi_sf\ra ds.
 \eeqlb
Furthermore, the solution is increasing in $(\Pi,f)\in
M_\rho(0,\infty)\times B_\rho(0,\infty)$ and $(\pi_t)_{t\ge 0}$ is a
semigroup of linear operators on $B_\rho(0,\infty)$ such that
 \beqlb\label{6.6}
\|\rho^{-1}\pi_tf\|\le (c - \la\Pi,\rho\ra)^{-1}\|\rho^{-1}f\|, \qquad
t\ge 0.
 \eeqlb
\eproposition

\proof For $f\in B_\rho^+(0,\infty)$ one can obtain (\ref{6.5}) by
differentiating both sides of (\ref{6.2a}), and (\ref{6.6}) follows by
(\ref{6.3}). By the linearity, the equation has a solution for any $f\in
B_{\rho}(0,\infty)$ and (\ref{6.6}) remains true. By
Proposition~\ref{t6.2} one can see the solution is increasing in
$(\Pi,f)\in M_\rho(0,\infty)\times B_\rho(0,\infty)$. The uniqueness of
the solution follows by a modification of the proof of Lemma~\ref{t6.1}.
\qed

\btheorem\label{tborel} There is a Borel right semigroup $(Q_t)_{t\geq0}$
on $N_{\rho}(0,\infty)$ defined by
 \beqlb\label{6.9.1}
\int_{N_\rho(0,\infty)} e^{-\la\nu,f\ra} Q_t(\mu,
d\nu)=e^{-\la\mu,U_tf\ra}, \quad f\in B^+_{\rho}(0,\infty),
 \eeqlb
where $(t,x)\mapsto U_tf(x)$ is the unique locally $\rho$-bounded
positive solution of (\ref{6.2a}). Furthermore, we have
 \beqlb\label{6.10.1}
\int_{N_\rho(0,\infty)} \la\nu,f\ra Q_t(\mu, d\nu) = \la\mu,\pi_tf\ra,
\qquad f\in B_\rho(0,\infty),
 \eeqlb
where $(t,x)\mapsto \pi_tf(x)$ is the unique locally $\rho$-bounded
solution of
 \beqlb\label{6.5.1}
\pi_tf(x)=f(x-t)1_{\{x>t\}} + c^{-1}\int_0^t 1_{\{x>t-s\}}
\la\Pi,\pi_sf\ra ds.
 \eeqlb
\etheorem

\proof Let $(U_t^{(k)})_{t\geq 0}$ be defined as in the last step of the
proof of Proposition~\ref{t6.2}. By Theorem~\ref{t2.1}, we can define a
Borel right semigroup $(Q_t^{(k)})_{t\geq 0}$ on $N(0,\infty)$ by
 \beqlb\label{3.11xx}
\int_{N(0,\infty)} e^{-\la\nu,f\ra} Q_t^{(k)}(\mu, d\nu)
 =
e^{-\la\mu,U_t^{(k)}f\ra}, \qquad f\in B^+(0,\infty).
 \eeqlb
In view of (\ref{2.11aa}) and (\ref{6.6}), if $\mu\in N_\rho(0,\infty)$
is a finite measure, we can regard $Q_t^{(k)}(\mu, \cdot)$ as a
probability measure on $N_\rho(0,\infty)$. Clearly, $N_\rho(0,\infty)$ is
a closed subset of $M_\rho(0,\infty)$ and the latter is an isomorphism of
$M(0,\infty)$ under the mapping $\nu(dx)\mapsto x\nu(dx)$. By
Theorem~1.20 of \cite{[Li11]} and the last step of the proof of
Proposition~\ref{t6.2} one can see (\ref{6.9.1}) really defines a
probability measure $Q_t(\mu, \cdot)$ on $N_\rho(0,\infty)$ for any
finite measure $\mu\in N_\rho(0,\infty)$. By approximating $\mu\in
N_\rho(0,\infty)$ with an increasing sequence of finite measures, we
infer the formula defines a probability kernel on $N_\rho(0,\infty)$.
Here (\ref{6.2a}) can be regarded as a special form of (6.11) in
\cite{[Li11]}. By Theorem 6.3 in \cite{[Li11]}, we can extend
$(Q_t)_{t\geq0}$ to a Borel right semigroup on $M_{\rho}(0,\infty)$. Then
we infer that $(Q_t)_{t\geq0}$ itself is a Borel right semigroup. The
moment formula (\ref{6.10.1}) can be obtained as in the proof of
Proposition~\ref{t2.1}. \qed

A Markov process in $N_\rho(0,\infty)$ with transition semigroup
$(Q_t)_{t\ge 0}$ defined by (\ref{6.2a}) and (\ref{6.9.1}) will be
referred to as a \textit{single-birth branching system} of particles with
\textit{offspring position law} $\Pi$. Clearly, when $\Pi$ is a finite
measure on $(0,\infty)$, this reduces to a special case of the model
introduced in the last section.

%%%%%%%%%%%%%%%%%%%%%%%%%%%%%%%%%%%%%%%

\section{Subordinators with negative drift}

\setcounter{equation}{0}

In this section, we give a description of the branching structures in
subordinators with negative drift.   Set
$$
 C^1(\mbb R)=\{f\in C(\mbb R): f \text{ is differentiable and has bounded derivative.}\}
$$
Let $c>0$ be a constant and let $\Pi$
be a $\sigma$-finite measure on $(0,\infty)$ satisfying $\la\Pi,\rho\ra<
c$. Suppose that $\{S_t: t\ge 0\}$ is a subordinator with negative drift
generated by the operator $A$ given by
 \beqlb\label{6.7.1}
Af(x) = \int^{\infty}_0 [f(x+z)-f(x)]\Pi(dz) - cf'(x), \qquad f\in
C^1(\mbb R).
 \eeqlb
We assume $S_0=a>0$. Our assumption implies that $S_t\to -\infty$ as
$t\to \infty$, so the hitting time
 \beqnn
\tau^-_0 := \inf\{t> 0: S_t\leq 0\}
 \eeqnn
is a.s. finite. For $t\geq0$ set
 \beqlb\label{4.2aa}
J(t):=\{u\in[0, \tau^-_0]: S_{u-}\leq t< S_{u}\}
 \eeqlb
with the convention that $S_{0-}=0$. Then we define the measure-valued
process
 \beqlb\label{XII}
X_t = \sum_{u\in J(t)}\dz_{S_{u}-t}, \quad t\ge 0.
 \eeqlb
It is easy to see that $X_0=\dz_a$.

\btheorem\label{Thmpoisson2} The process $\{X_t: t\ge 0\}$ is a
single-birth branching system in $N_\rho(0,\infty)$ with transition
semigroup $(Q_t)_{t\ge 0}$ defined by (\ref{6.2a}) and (\ref{6.9.1}).
\etheorem

\proof \textit{Step 1)} We first assume $\Pi(dz)$ is a finite measure on
$(0,\infty)$. In this case we clearly have $\mbf{P}\{\#J(t)<\infty$ for
all $t\ge 0\} = 1$. Let
 \beqnn
C(t) = \{u\in [0,\tau_0^-]: S_u=S_{u-}=t\} \quad\mbox{and}\quad
\zeta(t) = \#C(t).
 \eeqnn
We can write $C(t) = \{\tau_1(t),\cdots, \tau_{\zeta(t)}(t)\}$ by ranking
the elements in  increasing  order. Let $\tau_0(t) = 0$ and
 \beqnn
\sz_{i}(t) = \inf\{u\ge \tau_{i-1}(t): S_u>t\}, \quad i=1,2, \cdots,
\zeta(t).
 \eeqnn
Then it is easy to see that $J(t) = \{\sigma_1(t),\cdots,
\sigma_{\zeta(t)}(t)\}$ and
 \beqlb\label{Xcom1}
X_t = \sum_{i=1}^{\zeta(t)}\dz_{S_{\sz_i(t)}-t}, \quad t\ge 0.
 \eeqlb
In particular, we have $\zeta(t) = \#J(t)$. Write $M_t = \min_{0\le r\le
t} S_r$ and $L_t = S_t-M_t$. Set $\eta_0 = 0$ and for $k\ge 1$ define
inductively
 \beqnn
\zeta_k = \inf\{t>\eta_{k-1}: S_t\neq M_t\}
 \quad\mbox{and}\quad
\eta_k = \inf\{t>\zeta_k: S_t=M_t\}.
 \eeqnn

\includegraphics[width=1.0\textwidth, height=0.3\textheight]{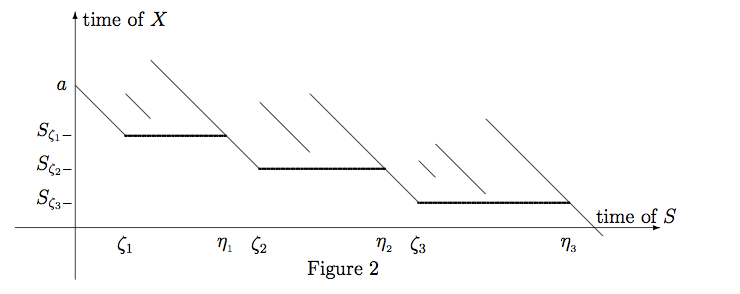}

\noindent It is clear that $a-S_{\zeta_1-}$ is an exponentially
distributed random variable with parameter $\lambda/c$, where $\lambda =
\Pi(0,\infty)$. By the memoryless property one can see $S_{\eta_{k-1}} -
S_{\zeta_k-} = S_{\zeta_{k-1}-} - S_{\zeta_k-}$ is also exponentially
distributed with parameter $\lambda/c$ for each $k\ge 1$. Let $e_k(t) =
(L_{\zeta_k+t} - L_{\zeta_k-})1_{\{t< \eta_k-\zeta_k\}}$ and let $F(dw)$
denote the distribution of $\{e_1(t): t\ge 0\}$ on $D^+[0,\infty)$, the
space of positive c\`{a}dl\`{a}g functions on $[0,\infty)$ equipped with
the Skorokhod topology. Then
 \beqlb\label{3.11aa}
(S_{\eta_{k-1}}-S_{\zeta_k-},\{e_k(t): t\ge 0\}), \quad k=1,2,\cdots
 \eeqlb
are i.i.d. random variables in $(0,\infty)\times D^+[0,\infty)$ with
 \beqlb\label{3.9aa}
\mbf{P}(S_{\eta_{k-1}}-S_{\zeta_k-}\in dy,e_k\in dw)
 =
\frac{\lambda}{c}e^{-\lambda y/c}dyF(dw), \quad y>0, w\in D^+[0,\infty).
 \eeqlb
It follows that
 \beqlb\label{3.12aa}
(S_{\zeta_k-},\{e_k(t): t\ge 0\}), \quad k=1,2,\cdots
 \eeqlb
are positioned in $(-\infty,a)\times D^+[0,\infty)$ as the atoms of a
Poisson random measure with intensity $c^{-1}\lambda dyF(dw)$. Let $n =
\max\{k\ge 0: \eta_k\le \tau_0^-< \zeta_{k+1}\}$. Then
$S_{\zeta_n-}{\overset{(d)}{=}}a\wedge \Theta$, where $\Theta$ is
exponentially distributed with parameter $\lz/c$. It is easy to see that
 \beqnn
X_t = \left\{\begin{array}{ll}
 \delta_{a-t} \quad &\mbox{for $0\le t< S_{\zeta_n-}$,} \cr
 \delta_{a-S_{\zeta_n-}} + \delta_{S_{\zeta_n}-S_{\zeta_n-}} \quad
 &\mbox{for $t = S_{\zeta_n-}$.}
\end{array}\right.
 \eeqnn
Therefore, the first offspring in the particle system is born at time
$S_{\zeta_n-}$. By (\ref{3.9aa}) we have
 \beqnn
\mbf{P}(S_{\zeta_n}-S_{\zeta_n-}\in dz)
 =
\mbf{P}(e_n(0)\in dz)
 =
\Pi(dz), \quad  z>0.
 \eeqnn
By the i.i.d. property of the random variables in (\ref{3.11aa}) we infer
that $\{X_t: t\ge 0\}$ is a branching system with parameters
$(g,c^{-1}\lambda,\lambda^{-1}\Pi)$, where $g(z)\equiv z$. In other
words, the system have transition semigroup defined by (\ref{6.2a}) and
(\ref{6.9.1}).

\textit{Step 2)} In the general case, let us consider an approximation of
the subordinator with drift. Let $\{N(ds,dz)\}$ be a Poisson random
measure on $(0,\infty)^2$ with intensity $ds\Pi(dz)$. Then a realization
of $\{S_t: t\ge 0\}$ is constructed by
 \beqnn
S_t := a+\int_0^t\int_0^\infty z N(ds,dz) - c t.
 \eeqnn
For each $k\ge 1$ we can define another subordinator with drift $\{S_t^{(k)}:
t\ge 0\}$ by
 \beqnn
S_t^{(k)} := a+ \int_0^t\int_{1/k}^{\infty} z N(ds,dz) - c t.
 \eeqnn
Then $S_t^{(k)}\leq S_t$ and as $k\to \infty$ we have
 \beqlb\label{6.11}
\sup_{0\le t\le T} (S_t^{(k)}-S_t) = S_T^{(k)}-S_T\to 0, \qquad T\ge 0.
 \eeqlb
Let $\Pi_k(dz) := 1_{\{ z \geq 1/k\}}\Pi(dz)$. Let $\{X_t^{(k)}: t\ge
0\}$ be the measure-valued process defined by (\ref{XII}) with $\{S_t:
t\ge 0\}$ replaced by $\{S_t^{(k)}: t\ge 0\}$. Then the first step
implies that $\{X_t^{(k)}: t\ge 0\}$ is a branching system in
$N(0,\infty)$ with transition semigroup $(Q_t^{(k)})_{t\ge 0}$ given by
(\ref{3.11xx}), where $(U_t^{(k)})_{t\geq 0}$ is defined as in the last
step of the proof of Proposition~\ref{t6.2}. Then we can also think of
$\{X_t^{(k)}: t\ge 0\}$ as a Markov process in $N_\rho(0,\infty)$. For
$t>t_n\ge t_{n-1}\ge \cdots\ge t_1\ge 0$ and $\{f,f_n,\cdots,
f_1\}\subset C_{\rho}^+(0,\infty)$, we have
 \beqlb\label{6.12}
&&\mbf{E}\exp\Big\{-\sum_{i=1}^n\la X_{t_i}^{(k)}, f_i\ra - \la
X_t^{(k)}, f\ra\Big\} \cr
 &&\qquad
=\,\mbf{E}\exp\Big\{-\sum_{i=1}^n\la X_{t_i}^{(k)}, f_i\ra - \la
X_{t_n}^{(k)}, U^{(k)}_{t-t_n}f\ra\Big\} \cr
 &&\qquad
=\,\mbf{E}\exp\Big\{-\sum_{i=1}^n\la X_{t_i}^{(k)}, f_i\ra - \la
X_{t_n}^{(k)}, U_{t-t_n}f\ra\Big\} + \varepsilon_k(f)
 \eeqlb
with
 \beqnn
|\varepsilon_k(f)|\le \mbf{E}\big|\exp\big\{-\la X_{t_n}^{(k)},
U_{t-t_n}f\ra\big\} - \exp\big\{- \la X_{t_n}^{(k)},
U^{(k)}_{t-t_n}f\ra\big\}\big|.
 \eeqnn
Let $(t,x)\mapsto \pi_tf(x)$ be the unique locally $\rho$-bounded
solution of (\ref{6.5}) and let $(t,x)\mapsto \pi_t^{(k)}f(x)$ be the
unique locally $\rho$-bounded solution of the equation with $\gamma$
replaced by $\gamma_k$. By Proposition~\ref{t6.3} and
Theorem~\ref{tborel},
 \beqnn
\varepsilon_k(f)
 &\le&
\mbf{E}\big\la X_{t_n}^{(k)},\big|U_{t-t_n}f -
U^{(k)}_{t-t_n}f\big|\big\ra \cr
 &=&
\pi^{(k)}_{t_n}\big|U_{t-t_n}f - U^{(k)}_{t-t_n}f\big|(a) \cr
 &\le&
\pi_{t_n}\big|U_{t-t_n}f - U^{(k)}_{t-t_n}f\big|(a).
 \eeqnn
By the proof of Proposition~\ref{t6.2} we have $U_tf(x) = \lim_{k\to
\infty}U_t^{(k)}f(x)$ increasingly. Then $\varepsilon_k(f)\to 0$ as $k\to
\infty$. From (\ref{6.12}) we get
 \beqnn
&&\mbf{E}\exp\Big\{-\sum_{i=1}^n\la X_{t_i}, f_i\ra - \la X_t, f\ra\Big\}
\cr
 &&\qquad
=\,\mbf{E}\exp\Big\{-\sum_{i=1}^n\la X_{t_i}, f_i\ra - \la X_{t_n},
U_{t-t_n}f\ra\Big\}.
 \eeqnn
The above equality can be extended to $\{f,f_n,\cdots, f_1\}\subset
B_\rho^+(0,\infty)$ by a monotone class argument. Then $\{X_t: t\ge 0\}$
is a Markov process in $N_\rho(0,\infty)$ with transition semigroup
$(Q_t)_{t\ge 0}$ given by (\ref{6.2a}) and (\ref{6.9.1}). \qed

%%%%%%%%%%%%%%%%%%%%%%%%%%%%%%%%%%%%%%%

\section{Negative subordinators with positive drift}

\setcounter{equation}{0}

In this section, we give a characterization of the branching structures
in negative subordinators with positive drift.
 We shall derive the result
from the one in the last section by a time reversal approach. Suppose
that $\Pi$ is a $\sigma$-finite measure on $(0,\infty)$ with
$\int_0^{\infty}1\wedge z\Pi(dz)< \infty$. Let $\{S_t^*: t\ge 0\}$ be a
L\'evy process generated by $A^*$ such that
 \beqlb\label{6.7}
A^*f(x) = \int^{\infty}_0 [f(x-z)-f(x)]\Pi(dz) + cf'(x), \qquad f\in C^1(\mbb R).
 \eeqlb
Assume $S_0^*=0$ and $0<c<\la\Pi,\rho\ra\leq \infty$. Then $S^*$ has
Laplace exponent
 $$
\psi(\beta)=c\beta-\int^{\infty}_{0}(1-e^{-\beta z})\Pi(dz), \qquad \beta\geq 0
 $$
Namely, ${\bf E}{e^{\beta S_t^*}} = e^{t\psi(\beta)}$. For $q\ge0$ let
$\Phi(q)=\sup\{t\geq 0: \psi(t)=q\}$. Define
$$
\tau^-_0:=\inf\{t> 0: S_t^*\leq 0\}.
$$
We have $\mbf{P}(0<\tau^-_0<\infty)=1$; see Corollary 5 in Section VII.1
of \cite{B96}. Then for $t\geq0$, set
$$
J(t):=\{u\in(0, T_0]: S_{u}^*\leq t< S_{u-}^*\}
$$
and define
 \beqlb\label{XI}
 X_t^*:=\sum_{u\in J(t)}\dz_{S_{u-}^*-t}.
 \eeqlb
with $X_0^*=\dz_{S^*_{\tau^-_0-}}$. Note that $\# J(t)<\infty, a.s.$

\btheorem\label{tborelc} There is a Borel right semigroup
$(Q_t)_{t\geq0}$ on $N_{\rho}(0,\infty)$ defined by
 \beqlb\label{6.9}
\int_{N_\rho(0,\infty)} e^{-\la\nu,f\ra} Q_t(\mu,d\nu)
 =
e^{-\la\mu,U_tf\ra}, \quad f\in B^+_{\rho}(0,\infty),
 \eeqlb
where $(t,x)\mapsto U_tf(x)$ is the unique locally $\rho$-bounded
positive solution of
 \beqlb\label{6.2}
U_tf(x)=f(x-t)1_{\{x>t\}} + c^{-1}\int_0^t 1_{\{x>s\}}ds \int^{\infty}_0[
1-e^{-U_{t-s}f(z)}]\Pi^+(dz),
 \eeqlb
where $\Pi^+(dz)=e^{-\Phi(0)z}\Pi(dz)$. Furthermore, we have
 \beqlb\label{6.10}
\int_{N_\rho(0,\infty)} \la\nu,f\ra Q_t(\mu, d\nu) = \la\mu,\pi_tf\ra,
\qquad f\in B_\rho(0,\infty),
 \eeqlb
where $(t,x)\mapsto \pi_tf(x)$ is the unique locally $\rho$-bounded
solution of (\ref{6.5}). \etheorem

\proof Note that $\beta_0 := \Phi(0)$ is the largest solution of
$\psi(\beta) = 0$. It follows that
 \beqnn
c - \int_0^\infty ze^{-z\Phi(0)}\Pi(dz) = \psi'(\Phi(0))> 0.
 \eeqnn
Then $(c,\Pi^+)$ satisfies the conditions of Theorem~\ref{tborel}. \qed

For reader's convenience, we first present a result on the distribution of
time reversed L\'evy processes which should be well-known to experts. For
$a>0$, let $\{S^{\#}_t: t\geq0\}$ and $\{S^{-}_t: t\geq0\}$ be two
subordinators with drift starting at $a>0$ with ${\bf E} e^{-\beta
S^{\#}_t} = e^{t\psi(\beta+\Phi(0))-a\beta}$ and ${\bf E}e^{-\beta
S^{-}_t} = e^{t\psi(\beta)-a\beta}$, respectively. Define
$T^{\#}(0)=\inf\{t\geq0: S^{\#}_t\leq0\}$ and $T^{-}(0)=\inf\{t\geq0:
S^{-}_t\leq0\}$. Note that $\mbf{P}\{T^{\#}(0)<\infty\}=1$.

\begin{lemma}\label{LLevy}
Given $S_{\tau^-_0-}^*=a$, the time reversed process
$\{S_{(\tau_0^--t)-}^*, 0\leq t<\tau_0^-\}$ has the same distribution as
$\{S^{\#}_t, 0\leq t< T^{\#}(0)\}.$
\end{lemma}
\proof Define $I_t=\inf\{0\wedge S_s^{*}: 0\leq s\leq t\}$ and
 $$
 J_t=\sum_{s\leq t}1_{\{S_s^*<I_{s-}\}}(S_s^*-S_{s-}^*).
 $$
For $a>0$, set $\varsigma(a)=\sup\{t\geq0: S_t^*-J_t \leq x\}$. By Lemma
21 and Theorem 17 in Chapter VII of \cite{B96}, conditioned on
$S_{\tau^-_0-}^*=a$,
 $$
\{S_t^*: 0\leq t< \tau_0^-\}\overset{(d)}{=}\{S_t^*-J_t: 0\leq t<
\varsigma(a)\}.
 $$
Then by Theorem 18 and Lemma 7 in Chapter VII of \cite{B96}, under
$\mbf{P}\{\cdot|S_{\tau^-_0-}^*=a\}$, $\{S_{(\tau_0^--t)-}^*: 0\leq
t<\tau_0^-\}$ has the same law as $\{S^{-}_t: 0\leq t< T^{-}(0)\}$ under
$\mbf{P}\{\cdot|T^{-}(0)<\infty\}$ which is the same as the law of
$\{S^{\#}_t: 0\leq t< T^{\#}(0)\}$. We have completed the proof. \qed

\btheorem\label{t6.4} The measure-valued process $\{X_t^*:t\geq0\}$
defined by (\ref{XI}) is a single-birth branching system with transition
semigroup $(Q_t)_{t\ge 0}$ determined by (\ref{6.9}) and
 \beqlb\label{6.9a}
\mbf{P}\{S(X_0^*)\in da\}
 =
c^{-1}{ e^{-\Phi(0)a}}\Pi([a,\infty))da\quad\text{for}\quad a>0,
 \eeqlb
where $S(X_0^*)=S^*_{\tau^-_0-}$ denotes the support for $X_0^*$.
 \etheorem

\proof (\ref{6.9a}) follows from Theorem 17 in Section VII of \cite{B96}.
 With the convention $S_{0-}^{\#} =
0$ we let
 \beqnn
J^{\#}(t) := \{u\in [0,T^{\#}(0)]: S_{t-}^{\#}\le u< S_t^{\#}\}.
 \eeqnn
For each $t\ge 0$ define the random measure $X_t^{\#}$ on $(0, \infty)$ by
 \beqnn
X_t^{\#} = \sum_{u\in J^{\#}(t)}\dz_{S_u^{\#}-t}.
 \eeqnn
Then by Theorem \ref{Thmpoisson2}, ${X^{\#}}$ is a Markov process with
transition semigroup $(Q_t)_{t\ge 0}$ given by (\ref{6.9}).

On the other hand, given $S_{\tau^-_0-}^*=a$, by Lemma \ref{LLevy}, we
have
 $$
\{S_{(\tau_0^--t)-}^*: 0\leq t<\tau_0^-\}
\overset{(d)}{=}\{S^{\#}_t: 0\leq t< T^{\#}(0)\}.
 $$
Thus given $S_{\tau^-_0-}^*=a$
 $$
\{X_t^*: 0\leq t<\infty \}\overset{(d)}{=}\{X_t^{\#}: 0\leq t<\infty\}.
 $$
We have completed the proof. \qed

%%%%%%%%%%%%%%%%%%%%%%%%%%%%%%%%%%%%%%%

\section{Properties of the branching systems}\label{Secproperty}

\setcounter{equation}{0}

In this section we discuss the properties of the measure-valued processes
via the exit problems for L\'evy processes. For a L\'evy process $S$ and
any $x\geq 0$ let
 \beqlb\label{Dtau1}
\tau^+_x=\inf\{t> 0: S_t>x\},\quad \tau^-_x=\inf\{t> 0: S_t\leq x\}
 \eeqlb
with the convention $\inf\emptyset=\infty$. Set $\mbf{P}_x\{\cdot\} =
\mbf{P}\{\cdot|S_0=x\}.$

\subsection{Properties of $X$}

In this subsection we discuss the properties of the measure-valued
process $X$ in Theorem~\ref{Thmpoisson2}, which is determined by process
$S$ which satisfies that $S_0=a$ and $S_t+ct$ is a subordinator with
L\'evy measure $\Pi$ and $\int_0^{\infty}z\Pi(dz)<c$. Recall
$\psi(\lambda)=c\lambda-\int_0^{\infty}(1-e^{-\lambda z})\Pi(dz)$ and
 $\Phi(q)=\sup\{t\geq 0: \psi(t)=q\}$ for $q\geq 0.$ Let $W$ denote
the scale function of $S$, i.e., an increasing and continuous function on $[0,
\infty)$ taking values in $[0, \infty)$ with
 $$
 \int_0^{\infty}e^{-\lambda x}W(x)dx=\frac{1}{\psi(\lambda)},
 $$
and we make the convention that $W(x)=0$ for $x<0$. We will need the
following solution to the two-sided exit problems.

\begin{lemma}\label{two-sided-exit}
For any $t\geq x, y\geq 0$ and $z>0 $,
 \[
{\mbf E}_x e^{-q\tau^-_0}=e^{-x\Phi(q)},
 \qquad
{\mbf P}_x\{\tau_0^-<\tau^+_t\}=\frac{W(t-x)}{W(t)}
 \]
and
 \beqnn
&&{\mbf P}_x\Big\{S_{\tau^+_t-}\in dy, S_{\tau^+_t}-t\in dz,
\tau^+_t<\tau_0^- \Big\} \\
 &&\qquad
= \Big(\frac{W(t-x)W(y)}{W(t)}-W(y-x)\Big)dy\Pi(t-y+dz).
 \eeqnn
\end{lemma}

\proof The first identity is from the beginning of page 212 of [13]. The
second identity follows by (8.8) of \cite{[Ky]} with $q=0$. The third
identity is (8.29) of \cite{[Ky]}. \qed

We first present a representation of $X_t$ for any fixed $t>0$.

\begin{Prop}\label{ProBrII1}
The random measure $X_t$ has the same distribution as
$\sum_{i=0}^{N-1} \delta_{Y_i}$, where $N$ and $(Y_i)$ are
independent random variables.
\begin{itemize}

\item For $a>t$,
 \begin{equation}\label{representA}
\mbf{P}\{N=n\}=\frac{1}{cW(t)}\left(1-\frac{1}{cW(t)}\right)^{n-1},
\quad n\geq 1,
 \end{equation}
$Y_0=a-t$ and $Y_i, i=1,2,...$ are i.i.d. random variables with
common distribution
 \begin{equation}\label{density}
\frac{1}{cW(t)}\int_0^t W(y)\Pi(t-y+dz)dy, \quad z>0.
 \end{equation}

\item For $ a\leq t$, $\mbf{P}\{N=0\}=W(t-a)/W(t)$ and
 \begin{equation}\label{representA1}
\mbf{P}\{N=n\}=\frac{1}{cW(t)}\left(1-\frac{W(t-a)}{W(t)}\right)
\left(1-\frac{1}{cW(t)}\right)^{n-1},
\quad n\geq 1,
 \end{equation}
$(Y_i)_{i\geq 1}$ are i.i.d. random variables with common
distribution (\ref{density}) and $Y_0$ is an independent random
variable with distribution
 \[
\int_0^t\left(\frac{W(t-a)W(y)}{W(t)}-W(y-a)\right)\Pi(t-y+dz)dy, \quad z>0.
 \]

\end{itemize}
\end{Prop}

\begin{proof}
Observe that by the construction, the total mass $X_t(0,\infty)$ is
exactly the total number of excursions above level $t$, which is the same
as the number of continuous downcrossings of level $t$. In addition, each
excursion of $S$ started with a jump upcrossing level $t$ has to come
back to level $t$ due to overall negative drift and lack of negative
jumps. Then (\ref{representA}) and (\ref{representA1}) follow easily from
the strong Markov property and Lemma \ref{two-sided-exit}.

For $t<a$, given $N=n\geq 1$, the excursion of $S$ above $0$ contains $n$
excursions at level $t$. The first excursion starts from $a$ and all the
excursions end at $a$. Further, by the strong Markov property the second to
the $n$th excursion starts with i.i.d. initial value $t+Y_1, \ldots,
t+Y_{n-1}$, respectively. By the construction the support of $X_t$ is
$\{a-t,Y_1,\ldots,Y_{n-1}\}$. Note that $Y_1$ is overshoot of the first
upward jump across level $t$. Then by Lemma \ref{two-sided-exit}
\begin{equation*}
\begin{split}
{\mbf P}_a \{Y_1\in dz\}
&={\mbf P}_t\{S_{\tau^+_t}\in t+dz, \tau^+_t<\tau_0^-\}\\
&=\frac{W(0)}{W(t)}\int_0^t W(y)\Pi(t-y+dz)dy.
\end{split}
\end{equation*}
The desired result follows. The corresponding result for $t\geq a$ follows
similarly.
\end{proof}

Our next result is on the weighted occupation time for $X$.

\begin{Prop}\label{occupation}
For any $f\in B_\rho^+(0,\infty)$ and $h\in B_\rho^+(0,\infty)$, we have
 \beqlb\label{occupation1}
{\mbf E} e^{-\int_0^\infty h(t)\lan X_t, f\ran dt}={\bf E} e^{-\lan
X_0,\omega_0\ran},
 \eeqlb
where $\omega$ is the unique nonnegative solution of the integral
equation
 \beqlb\label{eqn4.3} &&\omega_t(x)-
c^{-1}\int_t^\infty 1_{\{x>s-t\}}ds \int^{\infty}_0\Pi(dz)[
1-e^{-\omega_s(z)}]\cr&&\qquad\qquad=\int_t^\infty
h(s)f(x-s+t)1_{\{x>s-t\}}ds.
 \eeqlb
\end{Prop}

\begin{proof}
By Theorem \ref{t6.4}, similar to Section II.3 of Le Gall \cite{[Le99]}
we can show by induction together with the Markov property that for any
$0\leq t_1<\ldots<t_p$ and any $f_1,\ldots,f_p\in B_\rho^+(0,\infty)$,
\[{\mbf E} e^{-\sum_{i=1}^p \lan X_{t_i},f_i\ran}=e^{-\lan X_0,\omega_0\ran}\]
where $(\omega_t(x), t\geq 0, x\in (0,\infty))$ is the unique nonnegative
solution of the integral equation
 \[
\omega_t(x)- c^{-1}\int_t^\infty 1_{\{x>s-t\}}ds
\int^{\infty}_0\Pi(dz)[ 1-e^{-\omega_s(z)}]
=\sum_{i=1}^pf_i(x-t_i+t)1_{\{x>t_i-t\}}.
 \]
Further, by taking a limit on the Riemann sums we can show that
(\ref{occupation}) holds. Since the arguments for (\ref{eqn4.3}) is
similar to (\ref{6.2}), one could follow the proof of Corollary 9 in
Section II.3 of \cite{[Le99]} to get (\ref{occupation1}). We omit the
details here. \qed
\end{proof}

It is easy to recover Laplace transform for the total occupation time
$\int_0^\infty \lan X_t, 1\ran dt$. Observe that it is equal to the sum of
$a$ and sizes of all the jumps of $S$ up to time $\tau^-_0$, which is in
turn equal to $c\tau^-_0$. We then have
 \[
{\bf E}_a e^{-q\int_0^\infty \lan X_t, 1\ran dt}
={\bf E}_a e^{-qc\tau^-_0}=e^{-a\Phi(qc)}.
 \]

\subsection{Properties of $X^*$}

Properties of the measure-valued process $X^*$ in Theorem \ref{t6.4} can
also be investigated via the exit problems for process $S^*$ with
generator (\ref{6.7}), the negative of a subordinator with positive drift.

Throughout this subsection, for $q\geq0$, let $W^{(q)}$ be the scale
function for the spectrally negative L\'evy process $S^*$; i.e.;
$W^{(q)}(x)=0$ for $x<0$ and on $[0, \infty)$, it is an increasing and
continuous function taking values in $[0,\infty)$ with
 $$
 \int_0^{\infty}e^{-\lambda x}W^{(q)}(x)dx=\frac{1}{\psi(\lambda)-q},
 $$
for $\lambda>\Phi(q):=\sup\{\lambda\geq 0: \psi(\lambda)=q\}$, where
$\psi(\lambda)=c\lambda-\int_0^{\infty}(1-e^{-\lambda z})\Pi(dz)$. Write
$W$ for $W^{(0)}$. We also first present a result on the two-sided exit
problems of $S^*$; see Chapter 8 of \cite{[Ky]} and \cite{Be97}.

\begin{lemma}\label{two-sided-exitB}
For any $y>0>z$,
\[ {\mbf E} e^{-q\tau^-_0}=1-\frac{q}{\Phi(q)}W^{(q)}(0)\]
and
 \[
{\mbf P}\{S^*_{\tau^-_0-}\in dy, S^*_{\tau^-_0}\in dz\}
= W(0)e^{-\Phi(0)y}\Pi(dz-y)dy.
 \]
\end{lemma}

\begin{proof}
The first identity is just (8.6) of \cite{[Ky]}. The second identity the
equation right after (8.29) in \cite{[Ky]}. \qed
\end{proof}

For any $y> t$, let
\[g(y):={c^{-1}} e^{-\Phi(0)(y-t)}\int_0^t \Pi(y-dz)\frac{W(z)}{W(t)}\]
and
\[h(y):={c^{-1}}e^{-\Phi(0)(y-t)}\left\{\int_0^t
\Pi(y-dz)\left(1-\frac{W(z)}{W(t)}\right)+\Pi((y,\infty))\right\}.
\]
One will see from the proof of Proposition \ref{representB} that
$\int_t^\infty g(y)dy+\int_t^\infty h(y)dy=1$.

 Fix
$t>0$ until the end of the following Proposition \ref{representB}. We first
proceed to recover distribution for the total mass for $X_t^*$. The proof of
the following representation result is similar to Proposition
\ref{representA} and is omitted.

\begin{Prop}\label{representB}
The random measure $X_t^*$ has the same distribution as $\sum_{i=0}^{N-1}
\delta_{Y_i}$, where $N$ and $(Y_i)$ are independent random variables.
\begin{itemize}

\item
\begin{equation}\label{total_mass_1}
{\mbf P}\{N=0\}=1-\frac{1}{cW(t)}
\end{equation}
and for any $n\geq 1$
 \begin{equation}\label{total_mass_2}
{\mbf P}\{N=n\}=\frac{1}{cW(t)}\left(\int_t^\infty g(y)dy\right)^{n-1}
\int_t^\infty h(y)dy.
 \end{equation}

\item $Y_0$ has the density function $ h(t+y)/\int_t^\infty h(r)dr, y>
    0$ and $Y_i, i=1,2,\ldots$, share the common density function $
    g(t+y)/\int_t^\infty g(r)dr, y> 0$.

\end{itemize}
\end{Prop}

\proof Since $N=0$ if and only if the whole excursion of $S^*$ stays
below level $t$ up to time $\tau^-_0$, the probability
(\ref{total_mass_1}) just follows from Lemma \ref{two-sided-exit}.
Observe that the total mass $X^*_t(0,\infty)$ is exactly the number of
up-crossings (the same as the number of down-crossings) of level $t$ by
process $S^*$ until the time $\tau^-_0$. Each up-and-down-crossing of
level $t$ corresponds to an excursion starting at level $t$. All of such
excursions end at level $t$ except that the last one ends below $0$ at
time $\tau^-_0$, where the last excursion determines the residual life
time of a particle that can be either the ancestor or an offspring. Using
solutions to the two-sided exit problem in Lemmas \ref{two-sided-exit}
and \ref{two-sided-exitB} together with the strong Markov property
repeatedly at those up-crossing times of level $t$ we have
 \beqlb
&&{\mbf P}\{X_t^*(0,\infty)=n\}\cr &&={\mbf
P}\{\tau^+_t<\tau^-_0\}\left(\int_0^t{\mbf P}_t\{S^*_{\tau^-_t}\in
dz\}{\mbf P}_z\{\tau^+_t<\tau^-_0\}\right)^{n-1} \cr &&\qquad\quad
\times\left(\int_0^t{\mbf P}_t\{S^*_{\tau^-_t}\in dz\}{\mbf
P}_z\{\tau^+_t>\tau^-_0\}+ {\mbf P}_t\{S^*_{\tau^-_0}\leq 0\}\right)\cr
&&=\frac{W(0)}{W(t)}\left(\int_t^\infty e^{-\Phi(0)(y-t)}W(0)dy\int_0^t
\Pi(y-dz)\frac{W(z)}{W(t)}\right)^{n-1} \cr &&\qquad\times\int_t^\infty
e^{-\Phi(0)(y-t)}W(0)dy\left\{\int_0^t
\Pi(y-dz)\left(1-\frac{W(z)}{W(t)}\right)+\Pi((y,\infty)) \right\}.
 \eeqlb
Therefore, the probability (\ref{total_mass_2}) follows.

Given $X^*_t(0,\infty)=n $, the support of $X^*_t(0,\infty)$ consists of
those distances between the pre-down-crossing (of level $t$) values of
$S^*$ and $t$ for the $n$ excursions from $t$. By the strong Markov
property all these distances are independent. By Lemma
\ref{two-sided-exitB} the distances for the first $n-1$ excursions
following the same distribution of
 \begin{equation*}
 \begin{split}
&\int_0^t{\mbf P}_t\{S^*_{\tau^-_t-}\in t+dy,
S^*_{\tau^-_t}\in dz\}{\mbf P}_z\{\tau^+_t<\tau^-_0\}\left( \int_0^t{\mbf P}_t\{
S^*_{\tau^-_t}\in dz\}{\mbf P}_z\{\tau^+_t<\tau^-_0\}\right)^{-1}\\
&=e^{-\Phi(0)y}W(0)dy\int_0^t \Pi(t+y-dz)\frac{W(z)}{W(t)}
\left(\int_t^\infty g(r)dr\right)^{-1}\\
&=g(t+y)dy\left(\int_t^\infty g(r)dr\right)^{-1}.
 \end{split}
 \end{equation*}
The distance for the last excursion follows the distribution of
 \beqnn
&&\left(\int_0^t{\mbf P}_t\{S^*_{\tau^-_t-}\in t+dy, S^*_{\tau^-_t}\in
dz\}{\mbf P}_z\{\tau^+_t>\tau^-_0\}+ {\mbf P}_t\{S^*_{\tau^-_t-}\in t+dy,
S^*_{\tau^-_t}\leq 0\}\right)\\
&&\times\left(\int_t^\infty \int_0^t{\mbf P}_t\{S^*_{\tau^-_t-}\in t+dy,
S^*_{\tau^-_t}\in
dz\}{\mbf P}_z\{\tau^+_t>\tau^-_0\}+ {\mbf P}_t\{S^*_{\tau^-_t-}\in t+dy,
S^*_{\tau^-_t}\leq 0\}\right)^{-1}\\
&&=h(t+y)dy\left(\int_t^\infty h(r)dr\right)^{-1}.
 \eeqnn \qed

%\begin{Rem}
% Intuitively, the random variables
%$Y_1,\cdots ,Y_{n-1}$ represent the residual life times for $n-1$ of
%the offspring alive at time $t$ and $Y_0$ represents the residual
%life time either for the ancestor (if it is still alive at time
%$t$) or for an offspring (if the ancestor has died by time $t$).
%With probability $ \int_t^\infty
%e^{-\Phi(0)(y-t)}\Pi((y,\infty))dy/c\int_t^\infty h(y)dy $, $Y_0$
%represents the residual life time for the ancestor.
%\end{Rem}

Our next result is on the weighted occupation time for $X^*$. The
proof is similar to Proposition \ref{occupation} and is omitted.

\begin{Prop}
For any $f\in B_\rho^+(0,\infty)$ and $h\in B_\rho^+(0,\infty)$, we have
 \beqlb\label{occupation2}{\mbf E} e^{-\int_0^\infty h(t)\lan X^*_t, f\ran
dt}={\bf E} e^{-\lan X^*_0,\omega_0\ran},
 \eeqlb
where $\omega$ is the unique nonnegative solution of the integral
equation
 \beqnn
&&\omega_t(x)- c^{-1}\int_t^\infty 1_{\{x>s-t\}}ds
\int^{\infty}_0\Pi^+(dz)[
1-e^{-\omega_s(z)}]\cr&&\qquad\qquad=\int_t^\infty
h(s)f(x-s+t)1_{\{x>s-t\}}ds.
 \eeqnn
\end{Prop}

Observe that the total occupation time $\int_0^\infty \lan X^*_t, 1\ran
dt$ is just the sum of the sizes of all the jumps of process $S^*$ before
time $\tau^-_0$ together with $S^*_{\tau^-_0-}$. Further, this sum is
equal to $c\tau^-_0$ since $S^*_0=0$. By Lemma \ref{two-sided-exitB} we
then have
 \begin{eqnarray}
{\bf E} e^{-q\int_0^\infty \lan X^*_t, 1\ran dt}
 =
{\bf E} e^{-qc\tau^-_0}
 =
1-\frac{qc}{\Phi(qc)}W^{(qc)}(0)
 =
1-\frac{q}{\Phi(qc)}.
 \end{eqnarray}

%%%%%%%%%%%%%%%%%%%%%%%%%%%%%%%%%%%%%%%

\section{Connections with the CMJ model}

\setcounter{equation}{0}

Informally, the Crump-Mode-Jagers branching processes or the CMJ process
counts the size of a branching population system with random
characteristics. Informally, a particle, say $x$, of this process is
characterized by there random process
$$
(\lz_x, \zeta_x(\cdot), \omega_{x})
$$
which is an i.i.d. copy of $(\lz, \zeta(\cdot), \omega)$ and the
reproduction scheme is given in the following sense: if $x$ was born
at time $\sigma_x$, then
 \begin{enumerate}

\item $\lz_x$ is the life length of $x$;

\item $\zeta_x(\cdot)=\{0<\zeta_x^1<\zeta_x^2<\cdots<\lz_x\}$ is a
    point process defined on $(0,\lz_x)$.
    $\{\zeta_x^i+\sigma_x:i=1,\cdot\}$ is the collection of splitting
    times of $x$ at which it produces offspring.

\item $\omega_{x}^i$ is the number of children produced by $x$ at time
    $\sigma_x+\zeta_x^i$.

 \end{enumerate}
Let $Z(t)$ denote the total number of individuals in the system
at time $t$ with $Z(0)$ ancestors. In general, the process $\{Z(t): t\geq0\}$ is
not Markovian unless $\lz_x$ is exponentially distributed. Now assume
that
 \begin{enumerate}
\item The distribution of $\lz$ is determined by a probability measure
    $\eta(dx)$ on $(0,\infty)$;

\item $\zeta(\cdot)$ is a Poisson point process with parameter
    $\alpha$;

\item The distribution of $\omega^i$ is determined by a generating
    function $g(\cdot)$.

 \end{enumerate}
According to the argument in Section 2 and \cite{[DGL02]}, we may define
a measure-valued Markov process $Y=\{Y(t):t\geq0\}$ with transition
probabilities given by
 \beqlb\label{CJM2.10}
\int_{N(0,\infty)}e^{-\la\nu,f\ra}Q_t(\mu,d\nu)
 =
e^{-\la\mu,U_tf\ra},\qquad f\in B^{+}(0,\infty),
 \eeqlb
where $(t,x)\mapsto U_tf(x)$ is the unique locally bounded positive
solution of
 \beqlb\label{CJM2.11}
U_tf(x) = f(x-t)1_{\{x>t\}} + \az\int_0^t 1_{\{x> t-s\}} \big[1 -
g(\la\eta,e^{-U_sf}\ra)\big]ds.
 \eeqlb
Then the CMJ process $\{Z(t): t\geq0\}$ is just the total mass process of
$Y$; i.e. $Z(t)= \la Y(t),1\ra$.

\medskip

The connection between L\'evy processes and CMJ processes was first
investigated by Lambert in \cite{[L10]} which showed that the contour
process of a splitting tree defined from a suitable CMJ process is a
spectrally positive L\'evy process with negative drift killed when it hits
0. The starting position of the L\'evy process is just the life time of
the ancestor. Equivalently, given such a L\'evy process, one could
construct a CMJ process; see also \cite{[LSZ11]}. In those works, the
L\'evy measure, say $\gamma$, is assumed to be a $\sigma$-finite measure
on $(0,\infty]$ with $\int_{(0,\infty]}1\wedge z\gamma(dz)<\infty$.
 Our main result, Theorem
\ref{t6.2}, also gives similar relationships
between one-sided L\'evy processes of bounded variation and CMJ
processes.

\bigskip

\noindent\textbf{Acknowledgments.} Hui He wants to thank Concordia university
for his pleasant stay at Montreal where this work was done. We would like to
thank Professor Wenming Hong for his enlightening discussions. We also thank
Amaury Lambert for suggesting the time reversal treatment of the model in
Section~5.

\bigskip\bigskip

\noindent{\small Hui He and Zenghu Li: Laboratory of Mathematics and
Complex Systems, School of Mathematical Sciences, Beijing Normal
University, Beijing 100875, People's Republic of China. \\
\textit{E-mail:} {hehui@bnu.edu.cn} and {lizh@bnu.edu.cn}}

\bigskip

\noindent{\small Xiaowen Zhou: Department of Mathematics and Statistics,
Concordia University, 1455 de Maisonneuve Blvd. West,
Montreal, Quebec, H3G 1M8, Canada. \\
\textit{E-mail:} {xzhou@mathstat.concordia.ca}}


\bigskip\bigskip


\noindent\textbf{\large References}

\medskip

 \begin{enumerate}

 \renewcommand{\labelenumi}{[\arabic{enumi}]}

\bibitem{B96}{} J. Bertoin (1996): \textit{L\'evy processes},
    Cambridge University Press.

\bibitem{Be97}{} J. Bertoin (1997): Exponential decay and ergodicity
    of completely asymmetric L\'evy processes in a finite interval,
    \textit{Ann. Appl. Probab.} {\bf 7}, 156-169.

\bibitem{[DGL02]}{} D. A. Dawson, L. G. Gorostiza and Z. Li (2002):
    Non-local branching superprocesses and some related models,
    \textit{Acta Appl. Math.} {\bf 74}, 93-112.

\bibitem{[DuL02]}{} T. Duquesne and J.-F. Le Gall (2002):
    \textit{Random Trees, L\'{e}vy Processes and Spatial Branching
    Processes}, Ast\'{e}risque \textbf{281}.

\bibitem{[DuL05]}{} T. Duquesne and J.-F. Le Gall  (2005):
    {Probabilistic and fractal aspects of L\'{e}vy trees},
    \textit{Probab. Theory Relat. Fields} {\bf131}, 553-603.

\bibitem{[D91]}{} R. A. Doney (1991): Hitting probabilities for
    spectrally positive L\'evy processes, \textit{J. London Math.
    Soc.} {\bf s2-44}(3), 566-576.

\bibitem{[Dw75]}{} M. Dwass (1975): Branching processes in simple
    random walk, \textit{Proc. Amer. Math. Soc.} {\bf 51}, 270-274.

\bibitem{[HW09]}{} W. Hong and H. Wang (2010): Branching structure for
    an (L-1) random walk in random environment and its applications,
    \textit{arXiv:1003.3731v1}.

\bibitem{[HW10]} W. Hong and H. Wang (2010): Intrinsic branching
    structure within random walk on $\mathbb{Z}$, \textit{arXiv:
    1012.0636v1}.

\bibitem{[HZ10]}{} W. Hong and L. Zhang (2010): Branching structure
    for the transient (1;R)-random walk in random environment and its
    applications, \textit{Infin. Dimens. Anal. Quantum Probab. Relat.
    Top.} {\bf 13}, 589-618.

\bibitem{[KKS75]}{} H. Kesten, M.V. Kozlov, F. Spitzer (1975): A limit
    law for random walk in a random environment, \textit{Composit.
    Math.} {\bf 30}, 145-168.

\bibitem{[Key87]}{} E. S. Key (1987): Limiting distributions and
    regeneration times for multitype branching processes with
    immigration in a random environment, \textit{Ann. Probab.} {\bf
    15}, 344-353.

\bibitem{[Ky]}{} A. E. Kyprianou (2006): \emph{Introductory lectures
    on fluctuations of L\'evy processes with applications},
    Universitext. Springer, Berlin.

\bibitem{[L10]}{} A. Lambert (2010): The contour of splitting trees is
    a L\'evy process, \textit{Ann. Probab.} {\bf 38}, 348-395.

\bibitem{[LSZ11]}{} A. Lambert, F. Simatos, B. Zwart (2011): Scaling
    limits via excursion theory: Interplay between Crump-Mode-Jagers
    branching processes and Processor-Sharing queues, \textit{
    arXiv:1102.5620}.

\bibitem{[Le89]}{} J.-F. Le Gall (1989): Marches aleatoires, mouvement
    brownien et processus de branchement, \textit{Lect. Notes Math.}
    {\bf 1372}, 258-274.

\bibitem{[Le99]}{} J.-F. Le Gall (1999): \textit{Spatial branching
    processes, random snakes and partial differential equations},
    Birkh$\Ddot{\text{a}}$user.

\bibitem{[LL98a]}{} J.-F. Le Gall and J.-F. Le Jan (1998): {Branching
    processes in L\'evy processes: The exploration process},
    \textit{Ann. Probab.} {\bf26}, 213-252.

\bibitem{[LL98b]}{} J.-F. Le Gall and J.-F. Le Jan (1998): Branching
    processes in L\'evy processes: Laplace functionals of snake and
    superprocesses, \textit{Ann. Probab.} {\bf 26}, 1407-1432.

\bibitem{[Li11]}{} Z. Li (2011): \textit{Measure-Valued Branching
    Markov Processes}, Springer.

\bibitem{[NP89]}{} J. Neveu and J. W. Pitman (1989): The branching
    process in a Brownian excursion, \textit{Lect. Notes Math.} {\bf
    1372}, 248-257.

\bibitem{[S99]}{} K. Sato (1999): \textit{L\'{e}vy Processes and
    Infinitely Divisible Distributions}, Cambridge University Press.


\end{enumerate}
\end{document}